\documentclass{amsart}
\usepackage{amsthm, amssymb, amsmath, graphicx, amsfonts, pstricks, subcaption, enumitem, tipa, hyperref, todonotes, svg}
\DeclareUnicodeCharacter{2212}{-}

\theoremstyle{plain} \numberwithin{equation}{section}
\newcounter{dummy} 
\numberwithin{dummy}{section}
\newtheorem{theorem}[dummy]{Theorem}
\newtheorem{lemma}[dummy]{Lemma}
\newtheorem{cor}[dummy]{Corollary}
\newtheorem{proposition}[dummy]{Proposition}
\newtheorem{fact}[dummy]{Fact}
\newtheorem*{notation*}{Notation}
\newtheorem{definition}[dummy]{Definition}

\newcommand{\Mod}[1]{\ (\mathrm{mod}\ #1)}

\makeatletter
\newtheorem*{rep@theorem}{\rep@title}
\newcommand{\newreptheorem}[2]{%
	\newenvironment{rep#1}[1]{%
		\def\rep@title{#2 \ref{##1}}%
		\begin{rep@theorem}}%
		{\end{rep@theorem}}}
\makeatother
\newreptheorem{theorem}{Theorem}
\newreptheorem{lemma}{Lemma}

\begin{document}
\author{Ian M. Banfield}
\thanks{The author was supported by the SNF project no. 178756.}
\address{Mathematisches Institut, Universit\"at Bern, Siedlerstrasse 5, Bern, CH-3012}
\email{ian.matthew.banfield@gmail.com}
\markboth{Ian M. Banfield}{The maximum Levine-Tristram signature of torus knots}
\title{The Maximum Levine-Tristram Signature of Torus Knots}
\begin{abstract}
	We prove that the maximum of the Levine-Tristram signature function of a torus knot satisfies a reduction formula analogous to Gordon-Litherland-Murasugi's result for the classical signature. Applications include lower bounds for the topological 4-genus of torus knots.
\end{abstract}
\maketitle

\section{Introduction}
\label{sec:introduction}

Let $L \subset S^3$ be a link and let $\sigma_w(L) : \mathbb{S}^1 \to \mathbb{Z}$ be the Levine-Tristram signature function, defined as the algebraic signature of the weighted Seifert matrix
$$\sigma_w(L) = \mbox{sign}\left ( (1−w)A+ (1−\bar{w})A^\intercal \right ),$$
where $A$ is a Seifert matrix for $L$ \cite{MR1472978}. The classical signature of a link $L$ is $\sigma(L) = \sigma_{-1}(L)$, i.e. the signature of the Seifert pairing, represented by the symmetrized Seifert form. The Levine-Tristram signature function has been studied extensively, for details we refer to Conway's survey \cite{MR4294761}. In this work we investigate the peaks of the Levine-Tristram signature function.

\begin{definition}
	Let $L \subset S^3$ be a link and let $\sigma_w(L)$ be its Levine-Tristram signature function. The \textbf{maximum signature} of $L$ is $$\widehat{\sigma}(L) = \max_{w \in S^1} \sigma_w(L).$$
\end{definition}


By work of Nagel-Powell and Cha-Livingstone, it is known that the Levine-Tristram signature $\sigma_w(L)$ at $w \in \mathbb{S}^1$ is a topological concordance invariant if and only if $w$ does not arise as the root of a Laurent polynomial $p(t) \in \mathbb{Z}[t,t^{-1}]$ with $p(1) = \pm 1$ \cite{MR3609203, MR2054808}. A complex number $w \in \mathbb{S}^1$ which is a root of such a Laurent polynomial is called a Knotennullstelle. Since Knotennullstellen are isolated and the Levine-Tristram signature function $\sigma_w(L)$ is piecewise constant with finitely many discontinuities \cite{MR4294761}, it follows that the maximum signature $\widehat{\sigma}(L)$ is also a topological concordance invariant.

A torus knot is a knot that lies on the surface of an unknotted torus in $S^3$. More specifically, let $(p,q)$ be coprime and let the $(p,q)$ torus knot $T(p,q) = \widehat{\beta}$ be the closure of the braid $\beta = (\sigma_{p-1} \sigma_{p-2} ... \sigma_1)^q \in B_p$. A classical result from Gordon-Litherland-Murasugi is the following periodicity and reduction formula for the signature of torus knots:

\begin{fact}[\cite{MR617628}, Theorem 5.2, (I)]
	\label{eq:murasugi_recursive_signature}
	Let $0 < p, q$. The signature $\sigma$ of a torus knot $T(p,q)$ satisfies the following recursive relations. 
	\begin{align}
		\sigma(T(p,q+2p)) = \sigma(T(p,q)) + \begin{cases}
			p^2, & \text{for p even}\\
			p^2-1, & \text{for p odd}
		\end{cases}		
	\end{align}
\end{fact}

If $p$ is even, then Fact \ref{eq:murasugi_recursive_signature} can be slightly sharpened to give a formula for $T(p,q+p)$ rather than $T(p,q+2p)$, compare Proposition \ref{prop:periodicity_for_even_p}. In this work we prove an analogue of the Gordon-Litherland-Murasugi formula for the maximum signature, namely a reduction formula from $\widehat{\sigma}(T(p,q+p))$ to $\widehat{\sigma}(T(p,q))$. Unlike in the case of the classical signature (where the equivalent statement of Proposition \ref{prop:periodicity_for_even_p} for odd $p$ is false), the reduction formula for the maximum signature from $T(p,q+p)$ to $T(p,q)$ holds for both even and odd $p$:

\begin{proposition}
	\label{prop:mainthm}
	Let $0 < p < q$. The maximum signature $\widehat{\sigma}$ of a torus knot $T(p,q)$ satisfies the following recursive relation.
		$$
		\widehat{\sigma}(T(p,q+p)) = \widehat{\sigma}(T(p,q)) + \begin{cases}
			\frac{p^2}{2}, & \text{for p even}\\
			\frac{p^2-1}{2}, & \text{for p odd}
		\end{cases}		
		$$
\end{proposition}

The motivation for this work comes from the topological $4$-genus of torus knots. For 3-strand torus knots, Baader-B.-Lewark showed that $2 g_4(T(3,n)) = \widehat{\sigma}(T(3,n))$ and we asked whether this equality holds for all torus knots, \cite{MR4171377}. Specifically, the maximum signature plays the role of lower bound as $2 g_4(K) \geq \widehat{\sigma}(K)$ for knots $K \subset S^3$, by work of Powell \cite{MR3604490}.

\section{The Maximum Signature}
\label{sec:lt-signature}

We denote the Levine-Tristam signature at $w = e^{2\pi i t} \in \mathbb{S}^1$ by $\sigma_t(L)$ for $t \in [0,1]$. Our main technical tool for proving Proposition \ref{prop:mainthm} will be Litherland's formula for the Levine-Tristram signature of torus knots in terms of counts of lattice points \cite{MR547456}. Viro \cite{Viro_BranchedCoverings} proved that the Levine-Tristram signatures $\sigma_\frac{k}{m}(L)$ of a link $L \subset S^3$ (for $0 \leq k < m$) can be interpreted in terms of the intersection form on the $m$-fold cylic branched cover of a 4-dimensional manifold $(N, \partial N)$ over a properly embedded, null-homologous surface $(F, \partial F)$. (Here one requires that $N$ satisfies $H_1(N, \mathbb{Z}) = 0$ and that $\partial F = L \subset S^3 = \partial N$). Using Viro's characterization and work by Hirzebruch and Pham on the signatures of the Pham-Brieskorn manifolds, cf. \cite{Brieskorn1966}, Litherland derived the following expression for the Levine-Tristram signature of the $T(p,q)$.

\begin{proposition}[\cite{MR547456}, Proposition 1]
	If K is a (p,q) torus knot and $\xi = e^{2 \pi i t}$, $t$ rational, $0 < t < 1$, then $\sigma_{t}(K) = \sigma_{\xi^+} - \sigma_{\xi^-}$, where 
	\begin{enumerate}
		\item $\sigma_{\xi^+}$ is the number of pairs (i,j) of integers $0 < i < p$, $0 < j < q$, such that $t-1 < \frac{i}{p} + \frac{j}{q} < t \Mod{2}$, and
		\item $\sigma_{\xi^-}$ is the number of pairs (i,j) of integers $0 < i < p$, $0 < j < q$, such that $t < \frac{i}{p} + \frac{j}{q} < t+1 \Mod{2}$.
	\end{enumerate}
	\label{prop:litherland_signature}
\end{proposition}
The following expression for the Levine-Tristram signatures of torus knots $T(p,q)$ is immediate.
\begin{proposition}
	$\sigma_{t}(K) = 2 \sigma_{\xi^+} - (p-1)(q-1)$.
\end{proposition}

\begin{figure}
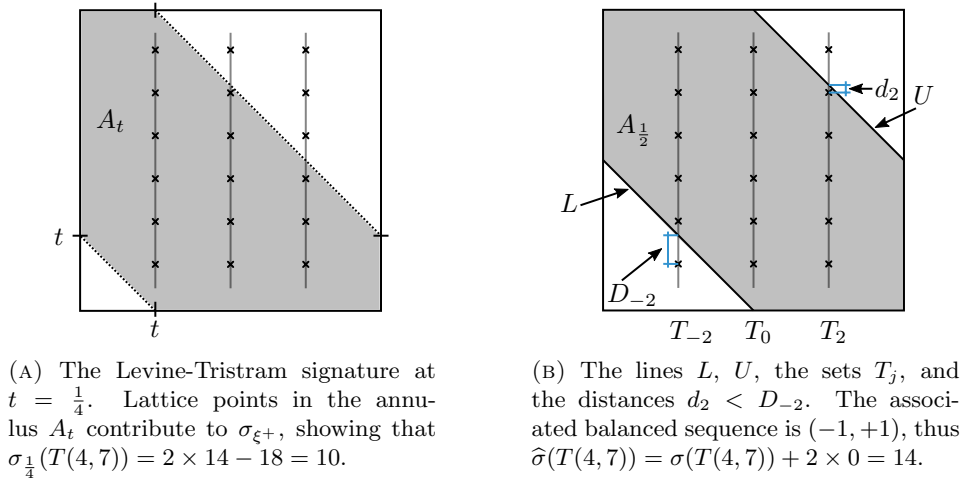

	\begin{subfigure}[t]{.45\textwidth}
		\centering
		\includesvg{lattice_example}
		\caption{The Levine-Tristram signature at $t = \frac{1}{4}$. Lattice points in the annulus $A_t$ contribute to $\sigma_{\xi^+}$, showing that $\sigma_{\frac{1}{4}}(T(4,7)) = 2 \times 14 - 18 = 10$. }
		\label{fig:lattice_example}
	\end{subfigure}\hfill
	\begin{subfigure}[t]{.45\textwidth}
		\centering
		\includesvg{lattice_terminology}
		\caption{The lines $L$, $U$, the sets $T_j$, and the distances $d_{2} < D_{-2}$. The associated balanced sequence is $(-1, +1)$, thus $\widehat{\sigma}(T(4,7)) = \sigma(T(4,7)) + 2\times 0 = 14$.}
		\label{fig:lattice_terminology}
	\end{subfigure}
	\caption{The lattice $\Sigma$ for the torus knot $T(4,7)$.}
\end{figure}

An example of Proposition \ref{prop:litherland_signature} to calculate the Levine-Tristram signature $\sigma_{\frac{1}{4}}(T(4,7))$ in terms of lattice point counts is given in Figure \ref{fig:lattice_example}.

Consider the rectangular lattice
\[
	\Sigma = \left\{\left (\frac{i}{p}, \frac{j}{q}\right) \in \mathbb{Q} \times \mathbb{Q} \right \} \cap (0,1) \times (0,1) \subset \mathbb{R}^2.
\]

A lattice point $a = (x,y) \in \Sigma$ contributes to $\sigma_{\xi^+}$ if and only if its Manhattan norm $d(a) = |x| + |y|$ satisfies $t < d(a) < t+1$, i.e. $a$ lies in the open Manhattan-norm annulus $A_t$, as is indicated in Figure \ref{fig:lattice_example}. Let $\Sigma^+ = \Sigma \cap A_t$ be the lattice points contributing to $\sigma_{\xi^+}$ ($\Sigma^-$ the lattice points contributing to $\sigma_{\xi^-}$, respectively). 

The lattice $\Sigma \subset I \times I$ is symmetric about the lines $x = \frac{1}{2}$ and $y = \frac{1}{2}$, and therefore symmetric under the reflection along the diagonals of $I \times I$. This implies 
\begin{equation}
	\label{eq:symmetry_of_LT_sig}
	\sigma_{\frac{1}{2} - \delta}(T(p,q)) =  \sigma_{\frac{1}{2} + \delta}(T(p,q)).
\end{equation}
An immediate consequence of the symmetry is that there exists $t \in [0,\frac{1}{2}]$ maximizing the Levine-Tristram signature.

For the remainder of this section, we introduce a coordinate system $(x,y)$ with origin at $\left (\frac{1}{2},0 \right ) \in \mathbb{R}^2$. Let  $$T_j = \left \{(x,y) \in \Sigma~|~x = \frac{j}{2p} \right \}$$
for $j$ satisfying $-p < j < p$ and $j \equiv p \Mod{2}$ be the lattice points grouped according to their $x$-coordinates. Clearly $\cup_{j}~T_j = \Sigma$.

\begin{notation*}
	For $x \in \mathbb{R}$, let $\lfloor x \rfloor$ be the greatest integer less than or equal to $x$, $\lceil x \rceil$ be the smallest integer greater than or equal to $x$, and let $\{x\} = x - \lfloor x \rfloor$ be the fractional part of $x$.
\end{notation*}

\begin{lemma}
	Let $0 < p < q$ and $(p,q) = 1$. Then there exists $t \in \left (\frac{1}{2} - \frac{1}{q}, \frac{1}{2} \right ]$ such that $\sigma_t(T(p,q))) = \widehat{\sigma}(T(p,q))$, i.e. the Levine-Tristram signature function attains a global maximum at $t$.
	\label{lem:max}
\end{lemma}
\begin{proof}
	From Figure \ref{fig:lattice_terminology}, we note that translating the annulus $A_t$ along the $x$-axis by $-\frac{1}{q}$ increases $\sigma_{\xi^+}$ by one for each $T_j$ with $j < 0$ and decreases $\sigma_{\xi^+}$ by one for each $T_k$ with $k \geq 0$. If $t \leq \frac{1}{2}$, then the number of sets $T_j$ with $j < 0$ is given by $\lfloor t  p \rfloor$, and therefore
	\[
		\sigma_{t-\frac{1}{q}}(T(p,q)) = \sigma_t(T(p,q)) + 2(\lfloor t  p \rfloor -  (p-\lfloor t  p \rfloor)) \leq \sigma_{t}(T(p,q)).
	\]
	The claim now follows from Equation \ref{eq:symmetry_of_LT_sig}.
\end{proof}

Let $L$ be the line $y = -x$ in our coordinate system ($U$ be the line $y = -x + 1$, respectively). Restricted to the unit rectangle $I \times I \supset \Sigma$ the line segments $L \cup U$ form the boundary of the annulus $A_{\frac{1}{2}}$. For $-p < j < 0$ and $0< k < p$ and $j \equiv k \equiv p \Mod{2}$, we define the following non-negative numbers (cf. Figure \ref{fig:lattice_terminology}): 

\begin{align}
	D_j &= 2pq \times \mbox{ minimal vertical distance from } (T_j \cap \Sigma^-) \mbox{ to } L, \\
	d_k &= 2pq \times \mbox{ minimal vertical distance from } (T_k \cap \Sigma^+) \mbox{ to } U.
	\label{eq:distances}
\end{align}
Up to scaling $D_j$ is the vertical distance from $T_j \cap \Sigma^-$ to the annulus $A_\frac{1}{2}$.

\begin{lemma}
	\label{lem:d_j_are_distinct_integers}
	The numbers $D_j, d_k$ are distinct integers and not equal to $p$.
\end{lemma}
\begin{proof}
	We start by proving that the numbers $D_j$ are integers. Let 
	$$(x,y) = \left (\frac{j}{2p}, \frac{k}{q} \right ) \in T_j \cap \Sigma^{-}$$
	be the lattice point in $T_j \cap \Sigma^{-} $ closest to $L$. The point $(x,y)$ realizes the minimum distance $D_j $ precisely when $k \in \mathbb{Z}_{\geq 0}$ is maximised. By definition of $D_j$,
	$$D_j = 2pq~(L(x) - y) = 2p \left (q L\left ( \frac{j}{2p} \right) - k \right ) = 2p~\left (\left \lfloor \frac{-jq}{2p} \right \rfloor + \left \{ \frac{-jq}{2p} \right \} - k\right ),$$
	and therefore $$D_j = 2p \left \{ \frac{-jq}{2p} \right \}.$$
	Thus $D_j$ is the unique integer $0 \leq D_j < 2p$ satisfying $D_j \equiv -jq \Mod{2p}$. We remark that
	$$k = \left \lfloor \frac{-jq}{2p} \right \rfloor$$ is the number of points in $T_j \cap \Sigma^-$ for $t = \frac{1}{2}$, this fact will be used in Lemma $\ref{lem:signature_as_lower_left_count}$.
	The numbers $d_k$ are integers as $D_j + d_{-j} = 2p$ by symmetry.
	We now turn our attention to uniqueness. Consider
	$$D_j \equiv D_{j'} \Mod{2p}~\Longleftrightarrow ~-jq \equiv -j'q \Mod{2p}~\Longleftrightarrow~2p|(-j + j')q,$$
	where $-p < j,j' < p$ and $j \equiv j' \equiv p \Mod{2}$.
	We claim that $2p|(-j + j')$. As $(p,q) = 1$, for $p$ even this follows from $(2p,q) = 1$, and for $p$ odd (and hence coprime to $2$) it follows from $p|(-j + j')$ and $2|(\pm j \pm j')$. The condition $-2p < \pm j \pm j' < 2p$ then implies that $j=j'$. We omit the other cases as they are similar. Lastly, if $D_j = p$ then $-jq = kp$, and so some prime of $p$ divides $q$ which contradicts coprimality.
\end{proof}
\newpage

Lemma \ref{lem:d_j_are_distinct_integers} allows one to associate a balanced sequence $(a_n) \subset \{ \pm 1 \}^{2m}$ of plus one and minus one to the numbers $\{D_j\}\cup \{d_k\}$, where $m = \lceil \frac{p}{2} \rceil - 1$, as follows: 
\begin{enumerate}
	\item Arrange the numbers $\{D_j\}\cup \{d_k\}$ in increasing order,
	\item Replace $D_j$ by $+1$ and $d_k$ by $-1$.
\end{enumerate}
For example, for $T(5,12)$ the values are as follows: $D_{-1} = 2$, $D_{-3} = 6, d_{1} = 8, d_{3} = 4$. Arranging these numbers in increasing order gives $D_{-1} < d_{3} < D_{-3} < d_{1}$, and the associated balanced sequence is $(+ 1, -1, + 1, -1) \in \{ \pm 1 \}^{4}$. 

The combinatorics of balanced sequences of $\pm 1$ was studied by Erdos-Kaplansky \cite{erdos1946sequences}. We need the following definition and Lemma on balanced sequences.

\begin{lemma}
	Let $(a_n) \subset \{\pm 1\}^{2m}$ be a balanced sequence of $\pm 1$, i.e. the number of $+1$ equals the number of $-1$. Consider the maximum of cyclical partial sums starting at $k$, i.e.
	\[
		M(k) = \max_{0 \leq l < \infty} \sum_{i=k}^{l} a_{i \Mod{2m}}.
	\]
	Then $M(k) = a_k + M(k+1)$.
	\label{lem:max_partial_sum_balanced_sequence}
\end{lemma}
\begin{proof}
\begin{align*}
		M(k+1)	&= \max_{0 < l < \infty} \left ( \sum_{i=k+1}^{l} a_{i \Mod{2m}} \right ) + a_k - a_k \\
				&= \max_{0 < l < \infty} \left ( \sum_{i=k}^{l} a_{i \Mod{2m}} \right ) - a_k \\
				&= M(k) - a_k.
\end{align*}
\end{proof}
Let $M(T(p,q))$ be the maximal cyclic partial sum $M(0)$ for the balanced sequence associated to the numbers $\{D_j\} \cup \{d_k \}$ for $T(p,q)$.

\begin{lemma}
	Let $0 < p < q$ and let $T(p,q)$ be a torus knot. The maximal signature is $\widehat{\sigma}(T(p,q)) = \sigma(T(p,q)) + 2 M(T(p,q))$.
	\label{lem:max_signature_based_on_sequence}
\end{lemma}
\begin{proof}
	Lemma \ref{lem:max} shows that the Levine-Tristram signature function has a global maximum in the interval $t \in \left (\frac{1}{2} - \frac{1}{q}, \frac{1}{2} \right]$. Consider the function
	$$f(s) = \frac{1}{2}(\sigma_{\frac{1}{2}-s}(T(p,q)) - \sigma(T(p,q))) \mbox{ for } s \in \left [0,\frac{1}{q} \right ).$$
	The function $f(s)$ counts, with sign, the difference between the number of lattice points that lie in the annuli $A_{\frac{1}{2}}$ and those that lie in $A_{\frac{1}{2}-s}$. This may be visualized as the number of lattice points ``gained" or ``lost" by the annulus on a translation along the $y$-axis by a distance of $-s$, starting at the annulus corresponding to the classical signature. The lattice points contributing with positive sign to $f(s)$ are precisely those in $T_j \cap \Sigma^{-}$ with distance $D_j < 2pqs$ to $L$ and the lattice points contributing with negative sign to $f(s)$ are those in $T_k \cap \Sigma^{+}$ with distance $d_k < 2pqs$ to $U$. All other lattice points either lie in all, or none, of the annuli $A_{\frac{1}{2}-s}$ and thus do not contribute. The order in which the lattice points are ``gained'' or ``lost'' is thus determined by the ordering of the numbers $D_j$ and $d_k$. This shows that the maximum of $f(s)$ on its domain is the maximal cyclical partial sum of the balanced sequence associated to the numbers $\{D_j\} \cup \{ d_k \}$.
\end{proof}
For example, we previously calculated that the associated balanced sequence for $T(5,12)$ is $(+ 1, -1, + 1, -1)$, so $M(T(5,12)) = 1$. Since the signature of $T(5,12)$ is $\sigma(T(5,12)) = 28$, Lemma \ref{lem:max_signature_based_on_sequence} shows that $\widehat{\sigma}(T(5,12)) = 28 + 2 = 30$.
\begin{cor}
	Let $0 < p < q$ and let $T(p,q)$ be a torus knot. The maximal signature satisfies $\widehat{\sigma}(T(p,q)) \leq \sigma(T(p,q)) + p - 1$.
\end{cor}
\begin{proof}
	The sequence associated to the numbers $\{D_j\} \cup \{d_k \}$ for $T(p,q)$ contains at most $\lceil \frac{p}{2}\rceil - 1$ plus ones, and thus $2 M(T(p,q)) \leq 2 (\lceil \frac{p}{2}\rceil - 1) \leq p-1$.
\end{proof}

In fact, the proof shows that the upper bound is $\widehat{\sigma}(T(p,q)) \leq \sigma(T(p,q)) + p - 2$ for even $p$. These bounds are sharp, as for the $T(p,2p+1)$ torus knot $D_j < p < d_i$ for all $i,j$, and so $M(T(p,2p+1))$ attains the upper bound.

Recall that in the proof of Lemma \ref{lem:d_j_are_distinct_integers} we showed that for the classical signature, the number of points in $T_j \cap \Sigma^-$ is given by $|T_j \cap \Sigma^-| = \lfloor \frac{jq}{2p} \rfloor$.

\begin{lemma}
	Let $0 < p < q$ and let $T(p,q)$ be a torus knot. The signature of $T(p,q)$ is
	\[
		\sigma(T(p,q)) = (p-1)(q-1) - 4 \sum_{\substack{0 < j < p\\j \equiv p \Mod{2}}} \left \lfloor \frac{jq}{2p} \right \rfloor.
	\]
	\label{lem:signature_as_lower_left_count}
\end{lemma}
\begin{proof}
	Note that by symmetry, $|T_j \cap \Sigma^-| = |T_{-j} \cap \Sigma^-|$, and therefore \\$\xi^{-} = |\Sigma^-| = 2 \sum_{j > 0} |T_j \cap \Sigma^-|$. The claim now follows from Proposition \ref{prop:litherland_signature}.
\end{proof}

We are now ready to establish Proposition \ref{prop:mainthm}. To begin, we prove a periodicity for the classical signature for $T(p,q)$ for even $p$.

\begin{proposition}
	Let $0 < p < q$, $p$ even, and let $T(p,q)$ be a torus knot. Then the signature satisfies
	\[
		\sigma(T(p,p+q)) = \sigma(T(p,q)) + \frac{p^2}{2}.
	\]
	\label{prop:periodicity_for_even_p}
\end{proposition}
\begin{proof}
	Using Lemma \ref{lem:signature_as_lower_left_count} we calculate (omitting the index set),
	\begin{align*}
		\sigma(T(p,p+q))	&= (p-1)(p+q-1) - 4 \sum \left \lfloor \frac{j(p+q)}{2p} \right \rfloor \\
							&= (p-1)(q-1) + (p-1)p - 4 \sum \left ( \left \lfloor \frac{jq}{2p} \right \rfloor + \frac{j}{2} \right ) \\
							&= \sigma(T(p,q)) + (p-1)p - 4 \sum_{i=1}^{\frac{p}{2}-1} i \\
							&= \sigma(T(p,q)) + \frac{p^2}{2}.
	\end{align*}
\end{proof}

\begin{proof}[Proof of Proposition \ref{prop:mainthm}, for even $p$]
	Let $D'_j, d'_k$ be the distances \ref{eq:distances} for $T(p,p+q)$, and $D_j, d_k$ for $T(p,q)$. Recall that $$D'_j = -j(p+q) \equiv -jq = D_j \mod{2p},$$ as $j$ is even. Thus $M(T(p,p+q)) = M(T(p,q))$ as the underlying balanced sequences coincide, and the claim follows from Lemma \ref{lem:max_signature_based_on_sequence} and Proposition \ref{prop:periodicity_for_even_p}.
\end{proof}

\begin{proof}[Proof of Proposition \ref{prop:mainthm}, for odd $p$]
	Let $D'_{-j}, d'_k$ be the distances \ref{eq:distances} for $T(p,p+q)$, and $D_{-j}, d_k$ for $T(p,q)$. As $j$ is odd and positive we have
	\begin{align*}
		D'_{-j} &= j(p+q) \equiv p + jq = p + D_{-j} \mod 2p, \\
		d'_k &= 2p -D'_{-k} \equiv -p - D_{-k} \equiv d_k + p \mod 2p.
	\end{align*}
	The numbers $D_{-j}, d_k$ are distinct by Lemma \ref{lem:d_j_are_distinct_integers} and satisfy $D_j + d_{-j} = 2p$, and thus $\frac{p-1}{2}$ of the numbers $D_j$, $d_k$ are greater than $p$. This shows that the balanced sequence for $\{D'_j\} \cup \{d'_k \}$ is obtained from the sequence associated to $\{D_j\} \cup \{d_k \}$ by a cylical shift of $\frac{p-1}{2}$. Lemma \ref{lem:max_partial_sum_balanced_sequence} then implies that
	\begin{align*}
		M(T(p,p+q)) &= M\left (\frac{p-1}{2} \right ) \\
					&= M(T(p,q)) - \sum_{i=1}^\frac{p-1}{2} a_i \\
					&= M(T(p,q)) - \underbrace{|\{D_{-j}~|~D_{-j} < p \}|}_{\mbox{\# +1 }} + \underbrace{\left ( \frac{p-1}{2}-|\{D_{-j}~|~D_{-j} < p\}| \right)}_{\mbox{\# -1}} \\
					&= M(T(p,q)) - 2 |\{D_{-j}~|~D_{-j} < p \}| + \frac{p-1}{2}.
	\end{align*}
	The classical signature $\sigma(T(p,p+q))$ can be calculated using Lemma \ref{lem:signature_as_lower_left_count}, as follows.
	\begin{align*}
		\sigma(T(p,p+q))	&= (p-1)(p+q-1) - 4 \sum \left \lfloor \frac{j(p+q)}{2p} \right \rfloor \\
							&= (p-1)(p+q-1) - 4 \sum \left (\left \lfloor \frac{jq+p}{2p} \right \rfloor + \frac{j-1}{2} \right).
	\end{align*}
	If $D_{-j} < p$ then $\left \lfloor \frac{jq+p}{2p} \right \rfloor = \left \lfloor \frac{jq}{2p} \right \rfloor$ and otherwise $\left \lfloor \frac{jq+p}{2p} \right \rfloor = \left \lfloor \frac{jq}{2p} \right \rfloor + 1$. A straightforward calculation then shows that
	\begin{align*}
		\sigma(T(p,p+q)) = \sigma(T(p,q)) - 4 |\{D_{-j}~|~D_{-j} > p \}| + \frac{(p-1)(p+3)}{2}.
	\end{align*}
	To complete, we note that $|\{D_{-j}~|~D_{-j} < p \}| + |\{D_{-j}~|~D_{-j} > p \}| = \frac{p-1}{2}$ and appeal to Lemma \ref{lem:max_signature_based_on_sequence}.
	\begin{align*}
		\widehat{\sigma}(T(p,p+q))	&= \sigma(T(p,p+q)) + 2 M(T(p,p+q)) \\
								&= \sigma(T(p,q)) - 4 |\{D_{-j}~|~D_{-j} > p \}| + \frac{(p-1)(p+3)}{2}\\
								&\quad + 2M(T(p,q)) - 4 |\{D_{-j}~|~D_{-j} < p \}| + p-1\\
								&= \widehat{\sigma}(T(p,q)) + \frac{p^2-1}{2}.
	\end{align*}
\end{proof}

\section{Applications and Examples}
In this short section we study some specific classes of torus knots and prove that the difference $|\widehat{\sigma}(K) - \sigma(K)|$ can be arbitrarily large.
\begin{theorem}
	Let $p > 0$.
	\begin{enumerate}
	\item 
	\label{prop:p_p_plus1}
	$
	\widehat{\sigma}(T(p,p+1)) = \sigma(T(p,p+1)) + \begin{cases}
	p-2, & \text{for p even}\\
	0, & \text{for p odd}
	\end{cases},
	$
	\item
	\label{prop:p_2p_plus_1}
	$\widehat{\sigma}(T(p,2p+1)) = p^2 + p - 2$.
	\end{enumerate}
\end{theorem}
\begin{proof}
	For Equation \ref{prop:p_p_plus1}, let $p$ be even. A calculation shows that $D_{-j} \equiv j \Mod{2p}$ and thus the ordering of the numbers $\{D_j\}\cup \{d_k\}$ is $$D_{-2} < \dots < D_{-p+2} < d_2 < \dots < d_{p-2},$$ implying that $M(T(p,p+1)) = \frac{p-2}{2}$. For $p$ odd, $D_{-j} \equiv p + j \Mod{2p}$ and the ordering of the numbers $\{D_j\}\cup \{d_k\}$ in this case is $$d_2 < \dots < d_{p-2} < D_{-2} < \dots < D_{-p+2}.$$ The claim now follows from Lemma \ref{lem:max_signature_based_on_sequence}.

	For Equation \ref{prop:p_2p_plus_1}, a straightforward calculation of the numbers $\{D_j\}\cup \{d_k\}$ and Lemma \ref{lem:max_signature_based_on_sequence} gives 
	$$
	\widehat{\sigma}(T(p,2p+1)) = \sigma(T(p,2p+1)) + \begin{cases}
		p-2, & \text{for p even}\\
		p-1, & \text{for p odd}
	\end{cases},
	$$
	and by Fact \ref{eq:murasugi_recursive_signature}, 
	$$
	\sigma(T(p,2p+1)) = \sigma(T(p,1)) + \begin{cases}
		p^2, & \text{for p even}\\
		p^2-1, & \text{for p odd}
	\end{cases}.
	$$
	As the signature of the unknot is zero, $\widehat{\sigma}(T(p,2p+1)) = p^2 + p - 2$ now follows.
\end{proof}
\begin{cor}
	The difference $|\widehat{\sigma}(K) - \sigma(K)|$ can be arbitrarily large.
\end{cor}

\section*{Acknowledgements}
The author would like to thank Sebastian Baader and Lukas Lewark for helpful discussions.

\bibliographystyle{halpha}
\bibliography{Max_LT_signature}

\end{document}